\documentclass[10pt,a4paper]{article}
\usepackage{mathrsfs}
\usepackage{amsfonts}
\usepackage{}

\usepackage[leqno]{amsmath}
\usepackage{graphicx}
\usepackage{latexsym}
\usepackage{amsmath,amsfonts,amssymb,amsthm,mathrsfs,euscript,makeidx,color}
\usepackage{enumerate}

\usepackage[colorlinks,linkcolor=blue,anchorcolor=green,citecolor=red]{hyperref}

\def\sqr#1#2{{\vcenter{\vbox{\hrule height.#2pt
				\hbox{\vrule width.#2pt height#1pt \kern#1pt \vrule width.#2pt}
				\hrule height.#2pt}}}}
\def\5n{\negthinspace \negthinspace \negthinspace \negthinspace \negthinspace }
\def\4n{\negthinspace \negthinspace \negthinspace \negthinspace }
\def\3n{\negthinspace \negthinspace \negthinspace }
\def\2n{\negthinspace \negthinspace }
\def\1n{\negthinspace }

\def\dbE{\mathbb{E}}

\def\dbP{\mathbb{P}}

\def\dbR{\mathbb{R}}

\def\dbT{\mathbb{T}}

\def\dbV{\mathbb{V}}

\def\sB{\mathscr{B}}

\def\sY{\mathscr{Y}}

\def\cA{{\cal A}}

\def\cH{{\cal H}}

\def\cJ{{\cal J}}

\def\cR{{\cal R}}

\def\cU{{\cal U}}
\def\cV{{\cal V}}

\def\BR{{\bf R}}

\def\BU{{\bf U}}

\def\BX{{\bf X}}

\def\ds{\displaystyle}

\def\ns{\noalign{\ss}}

\def\ss{\smallskip}
\def\ms{\medskip}

\def\q{\quad}
\def\qq{\qquad}

\def\({\Big (}
\def\){\Big )}
\def\[{\Big[}
\def\]{\Big]}

\def\d{\delta}
\def\e{\varepsilon}

\def\l{\lambda}

\def\t{\tau}

\def\bde{\begin{definition}\label}
	\def\ede{\end{definition}}
\def\be{\begin{equation}}
\def\bel{\begin{equation}\label}
\def\ee{\end{equation}}
\def\bt{\begin{theorem}\label}
	\def\et{\end{theorem}}
\def\bc{\begin{corollary}\label}
	\def\ec{\end{corollary}}
\def\bl{\begin{lemma}\label}
	\def\el{\end{lemma}}
\def\bp{\begin{proposition}\label}
	\def\ep{\end{proposition}}
\def\bas{\begin{assumption}\label}
	\def\eas{\end{assumption}}
\def\br{\begin{remark}\label}
	\def\er{\end{remark}}
\def\bex{\begin{example}\label}
	\def\ex{\end{example}}
\def\ba{\begin{array}}
	\def\ea{\end{array}}
\def\ben{\begin{enumerate}}
	\def\een{\end{enumerate}}

\def\square#1{\vbox{\hrule\hbox{\vrule height#1%
			\kern#1\vrule}\hrule}}
\def\rectangle#1#2{\vbox{\hrule\hbox{\vrule height#1%
			\kern#2\vrule}\hrule}}

\font\tenbb=msbm10 \font\sevenbb=msbm7 \font\fivebb=msbm5

\newfam\bbfam
\scriptscriptfont\bbfam=\fivebb \textfont\bbfam=\tenbb
\scriptfont\bbfam=\sevenbb

\newtheorem{theorem}{\indent Theorem}[section]
\newtheorem{definition}[theorem]{\indent Definition}
\newtheorem{proposition}[theorem]{\indent Proposition}
\newtheorem{corollary}[theorem]{\indent Corollary}
\newtheorem{lemma}[theorem]{\indent Lemma}
\newtheorem{remark}[theorem]{\indent Remark}
\newtheorem{example}[theorem]{\indent Example}

\newtheorem{assumption}[theorem]{\indent Assumption}

\makeatletter

\@addtoreset{equation}{section}
\makeatother

\def\bea{\begin{equation*}}
\def\eea{\end{equation*}}

\def\bel{\begin{equation}\label}
\def\eel{\end{equation}}

\def\ba{\begin{array}}
	\def\ea{\end{array}}
\newcommand{\ad}{&\!\!\!\displaystyle}

\def\({\Big (}
\def\){\Big )}
\def\[{\Big[}
\def\]{\Big]}
\def\q{\quad}
\def\qq{\qquad}

\def\d{\delta}
\def\e{\varepsilon}

\def\ms{\vspace{0.2cm}}
\def\ds{\displaystyle}

\def\ns{\noalign{\smallskip}}

\def\argmin{\mathop{\rm argmin}}

\begin{document}
\title{Time-inconsistent  Risk-sensitive  Equilibrium for Countable-stated Markov Decision Processes}
\author{Hongwei Mei\thanks{Department of Mathematics, The University of Kansas, Lawrence, KS 66045, U.S. (hongwei.mei@ku.edu).}}
\maketitle

\begin{abstract}
	This paper is devoted to solving  a time-inconsistent risk-sensitive control problem with parameter $\e$ and its limit case ($\e\rightarrow0^+$) for countable-stated Markov decision processes (MDPs for short). Since the cost functional is time-inconsistent, it is impossible to find a global optimal strategy for both cases. Instead, for each case, we will prove the existence of time-inconstant equilibrium strategies which verify some step-optimality. Moreover, we prove  the convergence of the so-called $\e$-risk-sensitive equilibria  and the corresponding value functions as $\e\rightarrow0^+$. 
\end{abstract}
\section{Introduction}

A Markov decision process (MDP for short) is a five-tuple $(\BX,\BU,\{U(x):x\in \BX\},Q,c)$ where $\BX$ is the state space, $\BU$ is the action set, $\{U(x):x\in \BX\} $ is feasible actions, $Q$ is the transition kernel and $c$ is the cost-per-stage function. For its wide application in different areas, it has been well studied in the last few decades.

 To measure different types of risk in different real models, people have raised different cost functionals (i.e. $c$) for MDPs.  In this paper we are interested in the risk-sensitive cases, i.e. given an appropriate policy $\pi=\{u_t\}$, the cost functional parameterized by $\e$ is defined as
$$\ds J^\e(x;\pi)=\e\log\dbE^{\e,\pi}_x[\exp(\e^{-1}\cJ)] \text{ and } \cJ:=\sum_{k=1}^Nc_k(X_k,u_k)+c_N(X_{N+1}).$$ 

Classical risk-sensitive MDPs have been intensively studied since the  seminal paper \cite{Howard1972}.
In particular the average cost criterion has attracted a lot of researchers since it is quite different from the classical risk neutral average cost problem (e.g. see \cite{Cavazos2000,Cavazos2011,Jask2007,Dima1999,Ghosh2014}). As far
as applications are concerned, for example,
where portfolio management is considered  in \cite{Biele1999}, where revenue problems are treated in \cite{Barz2007} and where the application of risk-sensitive control in finance can be found in \cite{Baue2011}. In recent years, some partially observable risk-sensitive MDPs are considered in \cite{Baue2017} and a class of risk sensitive MDPs with some certain costs  are investigated in \cite{Baue2013}.

For general finite $\e>0$  case, dynamic programming is an efficient method to find the optimal control and derive the equation for the cost functional under the optimal control.

If $\e\rightarrow \infty$, one can see that 
$$J^\e(x,\pi)=\dbE^{\e,\pi}_x(\cJ)+\frac{\e^{-1}} 2\dbV_x^{\e,\pi}(\cJ)+O(\e^{-2})$$
where $\dbV^{\e,\pi}_x(\cJ)$ is the variance of $\cJ$ under $\dbP^{\e,\pi}_x$. Thus for $|\e|$ being large, the control problem  is the so-called  variance minimization problem  in which it is to find an optimal strategy  minimizing the variance cost $\dbV_x^{\e,\pi}(\cJ)$ among the set of strategies under which the mean cost $\dbE^{\e,\pi}_x(\cJ)$ attains its minimum (see \cite{Kawai1987} for example).

While the case that $\e$ is small becomes totally different. When $\e\rightarrow0^+$,   the decision-makers are significantly sensitive with all possible  risks including rarely existed ones which were ignored before. As a consequence, one may assume that the transition rate of the dynamics  depends on $\e$ proportionally. For example,   the so-called small noise model for  stochastic differential equations is introduced in \cite{Flem2006}  where the Brownian motion is scaled by $\sqrt{\e}$. The author investigated the limit behavior the value function of risk-sensitive type as $\e$ tends to 0. The procedure to derive the convergence as $\e\rightarrow 0^+$   is also called vanishing viscosity method. Moreover, similar idea  applied to Markov chains can also be found in two-time scale problems (e.g. two-time scale Markov Chain in \cite{Yin2012}).    
In this paper, our main effort are devoted to  such case as well. Different from the model in  \cite{Flem2006}, our cost functional is parametrized by an additional discounting $\t$.  i.e. 
\bel{jCJ}\ds J_{\t,t}^\e(x;\pi)=\e\log\dbE^{\e,\pi}_{t,x}[\exp(\e^{-1}\cJ_{\t,t})] \text{ and } \cJ_{\t,t}:=\sum_{k=t}^Nc_{\t,k}(X^\e_k,u_k)+c_\t(X^\e_{N+1}).\eel
and the corresponding value function is 
$V_t^\e(x;\pi)=J^\e_{t,t}(x;\pi)$ where $\dbE^\pi_{t,x}$ is the conditional expectation on $X^\e_t=x$ under the policy $\pi$. 

If $c_{\t,t}(x,u)=c_t(x,u)$ is independent of $\t$, note that we have  following recursion,
$$V^\e_{t}(x;\pi)=c_t(x,u)+\e\log\dbE^{\e,\pi}_{t,x}\[\exp\big(\e^{-1} V^\e_{t+1}(X^\e_1;\pi)\big)\],$$ by applying the Bellman principle, one can see that the problem time-consistent, i.e. if an optimal control 
can be constructed for that (initial pair), then it will stay optimal hereafter. If   $c$ is in an exponential discounting form, i.e. $c_{\t,t}(x,u)=\lambda^{t-\t}c(x,u)$ for some $0<\l<1$,  we have 
$$V_t^\e(x,\pi)=c(x,u)+\e\log\dbE^{\e,\pi}_{t,x}[\exp(\e^{-1}\l\cJ_{t+1,t+1})].$$ Due to the non-linear structure on the right-hand side, the Bellman principle fails and  the time-consistency will be lost (some concrete calculation is made in Example \ref{example1}). We also notice that even if we take $\e\rightarrow 0$, the Bellman principle still fails because in the definition of cost function, the risk-sensitive parameter $\e$ is same as the rate $\e$ in the transition of Markov chain. Similarly  one can find the problem is time-inconsistent if $c$ is not exponential discounting. Therefore, different from a classical control problem, the risk-sensitive control problem is time-inconsistent if the cost functional exists a discounting factor $\t$, even if it is in an exponential discounting form. Hence, we are motivated to analyze the time-inconsistent risk-sensitive control problem for practical purpose.

In a time-inconsistent problem, the optimal control which minimizes the value function now doesn't stay optimal in future. The detailed calculation for a new recursion involving $\t$ can be seen in Section \ref{eequilirium}. To deal with time-inconsistency, we have to find a  time-inconsistent equilibrium which is  locally optimal only in some appropriate sense. After the breakthrough in \cite{Yong2012b} and \cite{Ekeland2008,Ekeland2010}, there are  lots of works on time-inconsistent control  concerning  MDPs and continuous-time models in the last decade (e.g. see
\cite{Hu2010,Qi2017,Yong2012a,Yong2012b,Wei2017,Bjok2017,Mei2017}). 
To the best  knowledge of the author, there are few works being concentrated on the convergence results for time-inconsistent control problems with risk-sensitive cost functional when the risk-sensitivity parameter $\e$ goes to $0$.  The paper is to fill this gap.

Compared to those previous works on time-consistent risk-sensitive problems investigated  such as \cite{Ghosh2014,Ghosh2017}, time-inconsistency brings new interesting features  and mathematical difficulties to work with. One of the main  mathematical difficulties  brought by time-inconsistency in general state space $\dbR^d$  lies in  the existence of time-inconsistent equilibrium strategies. For non-degenerate stochastic diffusions  in $\dbR^d$, the  existence and uniqueness  of time-inconsistent equilibrium  can be found in \cite{Yong2012a}.   While for degenerate case, the existence  is  still an open problem due to  the lack of first-order regularity of the viscosity solution for a degenerate  second-order HJB equation. More explicitly, for a time-inconsistent problem in the space of $\dbR^d$, the identification of time-inconsistent equilibrium  requires that the  HJB equation admits a classical  solution, which is not necessarily true for a degenerate problem. To avoid such  mathematical gap, most of the existed  works  are only concerned with the verification theorem (i.e. necessary conditions) for a strategy to be a time-inconsistent equilibrium (e.g. see \cite{Bjok2014,Bjok2017}).   While in our problem, such restriction is critical since  the diffusion of  $\e=0$ is  degenerate definitely. Thus in this paper, the dynamic  is assumed to be valued in a  countable-stated space with discrete topology. We hope to investigate the general cases in the future papers. 
\ss

In view of the developments, one would question why we should be concerned with controlled Markov chains with time-inconsistent and  risk-sensitive costs. There are several reasons for the works on such problems. Firstly, controlled Markov chains are the simplest  controlled Markovian systems which have a broad application in real life.  There are numerous systems that can be formulated as controlled Markov chains and/or Markov decision processes and the like. Thus considering such systems is not only necessary but has broader impact.
In addition,   controlled Markov chains  can be used to build numerical schemes for stochastic control problems.
Moreover, as introduced in previous paragraph, it is very complicated but required to take care of the regularity issues in time-inconsistent problems. As will be seen in this paper, treating controlled Markov chains valued countable-stated space with simple structures enables us to deal with the regularity issue effectively without complicated conditions. This together with aforementioned approximation may lead to future consideration of numerical approximation of time-inconsistent problems, which is of practical concerns.

\ss

Let $\BX$ be a space with countable  many states and the control space $\BU$ is a complete metric space with metric $|\cdot,\cdot|_U$. Without loss of generality, we suppose that $\BX$ be the set of integers. Let $M(\BX)$ be the set of all functions on $\BX$. 
$B(\BX)$ is the set of
 of functions  bounded from below. Write $P(\BX)$ be the set of all probability measures on $\BX$.
 A function $f\in M(\BX)$ is called inf-finite if the set $\{x\in\BX:f(x)\leq K\}$ has finite elements for all $K\in \BR$. Let $C(\BU)$ be the set of continuous functions on $\BU$. A function  $f$ on $\BU$ is called inf-compact if the set $\{u\in\BU:f(u)\leq K\}$ is compact for all $K\in \BR$ (i.e. the set of real numbers). 

\ss

The set of admissible policies $\Pi$ is assumed to be the collection of all  deterministic Markov policies, i.e.
$$\Pi=\{\pi=\tilde u_1\oplus\cdots\oplus\tilde  u_T: \tilde u_t=u_t(\cdot) \text{ is a measurable function from $\BX$ to $\BU$}\}.$$
Write $\dbT:=\{1,\cdots,T\}$ and  $\pi_{t}:=\tilde u_t\oplus\cdots\oplus \tilde u_{T}$. Here the notation $\tilde u$ means the strategy $u(\cdot)\in \cU$.
\ms
  
Given a deterministic policy $\pi\in\Pi$, the transition probability is
\bel{transition}\dbP^{\e,\pi}(X_{t+1}=j|X_t=i,X_{t-1},\cdots,X_1)=q^\e_{t}(j;i,u_t(i)).\eel
where $q^\e_{t}(j;i,u)\geq 0$ and $\sum_{j\in\BX}q^\e_{t}(j;i,u)=1$. 

\ms

For each $(\t,t)\in\dbT\times\dbT$, let $f_{\t,t}:\BX\times\BU\mapsto\BR$ and $g_{\t}:\BX\mapsto\BR$. Define the time-inconsistent $\e$-risk-sensitive cost functional by
\bel{cost}J^\e_{\t,t}(x;\pi_t)=\e\log\dbE^{\e,\pi_t}_{t,x}\exp\[\e^{-1}\(\sum_{s=t}^{T}f_{\t,s}(X_s,u_s(X_s))+g_\t(X_{T+1})\)\]\eel
and the value function at $t\in\dbT$ is
\bel{value}V^\e_t(x;\pi_t):=J^\e_{t,t}(x;\pi_t).\eel

We define the limit cost and value function as $\e\rightarrow0^+$ by
\bel{cost-0}J_{\t,t}(x;\pi_t)=\limsup_{\e\rightarrow0^+}J^\e_{\t,t}(x;\pi_t)
\eel
and 
\bel{value-0}V_t(x;\pi_t):=J_{t,t}(x;\pi_t).\eel

The dependence of the transition matrix on $\e$ is  to identify the transitions to  those states which happen rarely. For example, suppose that for some $j_0\in\BX$, $q_t^\e(j_0;i,u)=e^{-\e^{-1}p_{ij_0}(u)}$ for some an appropriate numbers $p_{ij}(u)>0$. When $\e$ is small, $j_0$ is a rare state which happens with probability 0 in the limit dynamic (i.e. the limit of the Markov chain as $\e\rightarrow0$). Thus in a classical optimization problem whose cost function is independent of $\e$, the state $j_0$ is ignored by the decision maker.  While for a risk-sensitive problem with a cost \eqref{cost}, such rare state can not be ignored even though $j_0$ disappears in  the limit dynamic. The  rate $\e$ in the transition matrix corresponds to risk-sensitivity rate $\e$ in the cost functional.
\ss

As we mentioned before, the dependence of the cost functions $f$ and $g$  on a discounting factor $\tau$ makes the problem time-inconsistent generally. Thus we will find a time-inconsistent equilibrium which satisfies some local optimality. The following is the definition for a time-inconsistent risk-sensitive equilibrium.

\begin{definition}{\rm (1)
		A $T$-step strategy
		$\pi^{\e,*}\in \Pi$ is called a {\it time-inconsistent $\e$-risk-sensitive equilibrium } ($\e$-equilibrium for short) if the following {\it step-optimality} holds
		\bel{opt1}J^\e_{t,t}(x; \pi^{\e,*}_{t})\leq J^\e_{t,t}(x;\tilde u\oplus \pi^{\e,*}_{t+1})\text{ for any } t\in\dbT,~\tilde u\in\cU.\eel
			Recall $\tilde u=u(\cdot)\in\cU$.

		(2) A $T$-step strategy
		$\pi^*\in \Pi$ is called a {\it time-inconsistent risk-sensitive equilibrium} if the following {\it step-optimality} holds
		\bel{opt2}J_{t,t}(x; \pi^*_{t})\leq J_{t,t}(x;\tilde u\oplus \pi^*_{t+1})\text{ for any } t\in\dbT,~\tilde u\in\cU.\eel

	}
\end{definition}
 From the definition, we can see that  we restrict us to the Markov policy only even though our problem is time-inconsistent. Actually, it is a natural consequence of the step-optimality.  From the detailed derivation of the equilibrium in Section 3 (for example see \eqref{nequi}), the step-optimal strategy in each step is in a feed-back form of  the step number $t$ and the state value $x_t$ only, independent of the past history $x_k$ for $k< t$. Thus we are only required to consider  the Markov strategies to guarantee the step-optimality. The readers are also referred to \cite{Yong2012b} for more details.

From the step-optimality \eqref{opt1} and \eqref{opt2},  provided  all strategies after $k$th step (i.e. $\pi_{k+1}^*$), the $k$th-step strategy $u_k$ is the optimal strategy in the step under the cost functional with the discounting factor $\t=k$.   If we suppose different players take actions in different steps, the $k$th player makes his optimal strategy to minimize the cost functional under the discounting factor $\t=k$, given  the strategies of the players thereafter. 
Our main goal in the paper is to derive the time-inconsistent $\e$-risk-sensitive equilibrium and time-inconsistent risk-sensitive equilibrium. Moreover, we will prove that the convergence of time-inconsistent  $\e$-equilibria to time-inconsistent  equilibrium as $\e\rightarrow 0^+$.

The paper is arranged as follows. In Section \ref{sec:PR}, we will review some  results for LDP which will be used in our paper and present some preliminary lemmas. In Section \ref{sec:TE}, we will  derive the time-inconsistent   risk-sensitive equilibria
 and the corresponding recursive Hamiltonian sequences for both cases. Then in Section \ref{sec:con}, we prove the convergence of $\e$-equilibria as $\e\rightarrow0^+$. Finally, two illustrative examples are presented in Section \ref{sec:exp} and some concluding remarks are made in Section \ref{sec:conrem}.
\section{ Preliminary Results}\label{sec:PR}

\subsection{Large Deviation Principle}

In this subsection, we will review some well-known results on large deviation principle. On a complete separable space $\sY$, $I:\sY\mapsto[0,\infty]$ is called a (good) rate function if it is inf-compact. Let $Y_n$ be a sequence of $\sY$-valued random variables on some appropriate probability space.
$\{Y_n\}$ is said to satisfy the LDP  with rate function $I$ if

(1) for any closed subset $C$ of $\sY$, 
$$\limsup_{n\rightarrow\infty}\frac1n\log\dbP(Y_n\in C)\leq-\inf_CI.$$

(2) for any open subset $O$ of $\sY$, 
$$\liminf_{n\rightarrow\infty}\frac1n\log\dbP(Y_n\in O)\geq-\inf_OI.$$

Roughly speaking, the large deviation principle concerns with the rate of probability to zero for rare events. Thus the corresponding  risk-sensitive  problem is a certain type of   robustness control problems. Now  let's recall some   results on LDP which will be used in our paper.  For more details and their proofs, one can check \cite{PDRE1997}.

\begin{theorem}\label{pret} {\rm (1)}
	$\{Y_n\}$ satisfies the {\it LDP } with rate function $I$ if and only if $I$ is a rate function (i.e. inf-compact) and for any $h\in C_b(\sY)$ (i.e.  bounded continuous functions on $\sY$),
	$$\limsup_{n\rightarrow\infty}\frac1n\log\dbE\(\exp[nh (Y_n)]\)=\sup_{\sY}[h-I]$$
	
	{\rm (2)}
	$\{Y_n\}$ satisfies {\it LDP } with rate function $I$ if and only if $\{Y_n\}$ is exponential tight, i.e. for any $a>0$, there exists a compact subset $K_a$ of $\sY$ such that
	$$\frac1n\log\dbP(Y_n\in K_a^c)\leq -a$$
	 and for any bounded continuous function $h$ on $\sY$,
	$$\limsup_{n\rightarrow\infty}\frac1n\log\dbE\(\exp[nh (Y_n)]\)=\sup_{\sY}[h-I]$$
	
	{\rm (3)} If there exists a positive, inf-compact function $\cV$ on $\sY$ (i.e. Lyapunov function) satisfying
	\bel{lya}\sup_n\frac1n\log\dbE\(\exp[n\cV (Y_n)]\)<\infty,\eel
	then $\{Y_n\}$ is exponential tight.
	
	{\rm (4)} Let $P(\sY)$ be the set of probability measures on $\sY$. The following variational equality (i.e. Varadhan's equality) holds,
	\bel{entropy}\log\int_{\sY} e^h d\mu=\sup_{\nu\in P(\sY)}\(\int_\sY h d\nu- \cR(\nu\Vert\mu)\),\q\text{ for any }h\in C_b(\sY)\eel
	where the {\it relative entropy} $\cR(\cdot\Vert\cdot)$ is defined by 
	$$\cR(\nu\Vert\mu):=\int_{\sY}\log\(\frac{d\nu}{d\mu}\)d\nu,\q\mu,\nu\in P(\sY) .$$
	Moreover, if \eqref{lya} holds, then  \eqref{entropy} holds for any $h\in o(\cV)$ (i.e. $\lim_{|y|\rightarrow\infty}|h(y)|/\cV(y)=0$.)
\end{theorem}

\subsection{Preliminary Lemmas}
Let's recall the transition probability 
$$\dbP^{\e,\pi}(X_{t+1}=j|X_t=i, X_{t-1},\cdots,X_1 )=q^\e_{t}(j;i,u_t(i)).$$

For each $t\in\dbT$ and $\e>0$,	define $\Lambda^\e_t,\Lambda_t:\BX\times\BU\times B(\BX)\mapsto \BX$ by
\bel{Lambdae}\Lambda^\e_t(x,u;h):=\e\log\(\sum_{z\in\BX}\exp\left\{\e^{-1} h(z)\right\}q^\e_{t}(z;x,u)\)\text{ and }\Lambda_t(x,u;h)=\lim_{\e\rightarrow 0^+}\Lambda^\e_t(x,u;h).\eel
Note that $h$ is bounded,  $\Lambda^\e_t(x,u;h)$ is  well-defined.  $\Lambda_t(x,u;h)$ is well-defined because of the following assumption.
\ms

{\bf Assumption (A)}: {
	{\rm(A1)} There exists an inf-finite, positive function $\cV:\BX\mapsto \BR$ such that for each $(x,u)\in\BX\times\BU$,
	$$\sup_{0<\e\leq \e_0}\Lambda^\e_t(x,u;\lambda_0\cV)<\infty,\q \text{for some } \lambda_0,\e_0>0.$$		
	
	{\rm(A2)} Given any $t\in\dbT$ and  $h\in B(\BX)$, $\Lambda_t^\e(x,\cdot;h)$ is a continuous function of  $u\in\BU$. Moreover for each $(x,u)\in\BX\times\BU$, there  exists a  rate function $I_t(\cdot;x,u):\BX\mapsto \BR$ such that for any $h\in B(\BX)$, and $u^\e\rightarrow u$ in $
	\BU$
	\bel{rate}\lim_{\e\downarrow0}\Lambda^\e_t(x,u^\e;h)=\sup_{z\in\BX}[h(z)-I_t(z;x,u)]=\Lambda_t(x,u;h).\eel

	{\rm(A3)} There exists a $\lambda_0>0$ and a  constant $K_u$ depending on $u$ only such that for any $\lambda\in(0,\lambda_0)$ and each $u\in \BU$,
	$$\limsup_{|x|\rightarrow\infty}\frac{\sup_{0<\e\leq\e_0}\Lambda^\e_t(x,u;\lambda\cV)}{\lambda\cV(x)}< K_u.$$

}

For the positive function $\cV$ on $\BX$ in (A1),  we define a subset $B_\cV(\BX)$ of $B(\BX)$ by
$$B_\cV(\BX):=\{h\in B(\BX):\limsup_{|x|\rightarrow\infty}\frac{h(x)}{\cV(x)}=0\}.$$
We also write  
$$\sB_\cV:=\left\{\{h^\e\}\subset   B_\cV (\BX): h_\e \text{ is uniformly bounded below and } \sup_\e h^\e\in   B_\cV (\BX) \right\}.$$

\begin{remark}\label{remarkA}
	{\rm (1)  By Theorem \ref{pret}, (A1) and (A2) are sufficient for that $\{X_{t+1}^\e|X_{t}^\e=x,u\}$ satisfies LDP with rate function $I_t(\cdot;x,u)$. Moreover for $h\in   B_\cV (\BX)$, $\Lambda_t^\e(x,u;h)$ and $\Lambda_t(x,u;h)$ are well-defined  and \eqref{rate} holds as well.
		
		(2) (A2) says that the rate function $I_t$ is uniform on any compact subset of $\BU$. We can conclude that $\Lambda_t^\e(x,u;)$ converges to $\Lambda_t(x,u;h)$ uniformly on any compact set of $\BU$. Moreover, $\Lambda_t(x,u;h)$ is continuous on any compact subset of $\BU$ given fixed $x$ and $h$ (See Proposition 1.2.7 in \cite{PDRE1997}).  
	
(3) If (A1) and (A2) hold, the definition of $\Lambda^\e_t(x,u;h)$ and $\Lambda_t(x,u;h)$ can be extended to all $h\in   B_\cV (\BX)$ and (A2) is true for all $h\in   B_\cV (\BX)$. }
\end{remark}
\ms

 In this paper,  $  B_\cV (\BX)$ is equipped with the following metric,
$$w(h,h'):=\sup_\BX\frac{|h-h'|}{\cV}.$$
The following lemma says that $(  B_\cV (\BX),w)$ is a complete metric space.
\begin{lemma}\label{lemmapoint}Given $\cV$ defined in {\rm (A1)}, the followings hold.

	{\rm (1)} If $h_n\in   B_\cV (\BX)$ with $w(h_n,h_m)\rightarrow 0$ for any $n,m\rightarrow\infty$,  there exists a $h\in   B_\cV (\BX)$ such that
	$w(h_n,h)\rightarrow 0$, i.e.  $(  B_\cV (\BX),w)$ is a complete metric space.\ss
	
	{\rm (2)} If   $h_n$ is uniformly bounded below with $\sup_n h_n\in   B_\cV (\BX)$, then $\{h_n\}$ has a convergent subsequence in $(  B_\cV (\BX),w)$. As a result, if  $\{h_n\}\in\sB_{\cV}$  and $h_n$ converges to $h$ point-wisely, then $h_n$ converges to $h$ in $(  B_\cV (\BX),w)$ .
\end{lemma}

\begin{proof} (1) For such $\{h_n\}$, it is easy to see that there exists a $h\in M(\BX)$ such that $h_n$ converges to $h$ point-wisely. Now we show that the convergence is in metric sense as well.
	
	 For any $\d>0$, there exists a $N_\d>0$ such that
	$$w(h_n,h_m)<\d,\text{ for any } n,m\geq N_\d.$$
	Note that for any $m\geq N_\d$, 
	$$\ba{ll}\ds\limsup_{|x|\rightarrow\infty}\frac{|h(x)|}{\cV(x)}\ad\leq \limsup_{|x|\rightarrow\infty}\(\frac{|h_m(x)|}{\cV(x)}+\lim_{n\rightarrow\infty}\frac{|h_n(x)-h_m(x)|}{\cV(x)}\)\\
	\ns\ad\leq \lim_{n\rightarrow\infty}\sup_{x\in\BX}\frac{|h_n(x)-h_m(x)|}{\cV(x)}<\d.\ea$$
	By the arbitrariness of $\d>0$, we have $h\in   B_\cV (\BX)$. 
	
	For any fixed  $\d>0$, let $n_k$ satisfy  
	$$w(h_n,h_m)<\frac \d{2^k},\text{ for any } n,m\geq n_k.$$
	Then one can easily see that
	$$\sum_{k=1}^\infty w(h_{n_{k+1}},h_{n_{k}})<\d.$$
	It follows that for any $n>N_\d$.
	$$\ba{ll}\ds\sup_{x\in\BX}\frac{|h(x)-h_n(x)|}{\cV(x)}\ad\leq \sup_{x\in\BX}\sum_{k=1}^\infty\frac{|h_{n_{k+1}}(x)-h_{n_k}(x)|}{\cV(x)}+\sup_{x\in\BX}\frac{|h_{n}(x)-h_{n_1}(x)|}{\cV(x)} \\
	\ns\ad \leq\sum_{k=1}^\infty w(h_{n_{k+1}},h_{n_{k}})+w(h_{n_{1}},h_{n})\leq2\d.\ea$$
	It is equivalent to say
	$$\lim_{n\rightarrow\infty }w(h_n,h)=0.$$
	
	(2) By the hypothesis, one can easily see that $\{h_n\}$ has a point-wisely convergent subsequence with limit $h$. We still write the subsequence as $\{h_n\}$. Obviously we have 
	$h\in   B_\cV (\BX)$ since $h_n$ is uniformly bounded below and  $ h\leq \sup_n h_n\in   B_\cV (\BX)$.
	
	Note that for any $\d>0$, there exists a $x_\d>0$ such that 
	$$\frac{\sup_nh_n(x)}{\cV(x)}\leq \d\text{ for } x\geq x_\d.$$
	Then by the point-wise convergence, it follows that
	$$\lim_{n\rightarrow\infty}\sup_{x\in\BX}\frac{|h_n(x)-h(x)|}{\cV(x)}\leq \lim_{n\rightarrow\infty}\sup_{|x|\leq x_\d}\frac{|h_n(x)-h(x)|}{\cV(x)}+2\d=2\d.$$
	By the arbitrariness of $\d>0$, we have 
	$$\lim_{n\rightarrow\infty}w(h_n,h)=0.$$
	\end{proof}

\ms

Now we first prove that well-posedness of $\Lambda$ and $\Lambda^\e$ on the space $  B_\cV (\BX)$.

\begin{lemma}\label{corlambdah} Under Assumption {\rm (A)}, for  any $\{h^\e\}\in   \sB_\cV$ and each $u\in\BU$,
	$\{\Lambda^\e_t(\cdot,u;h^\e)\}\in \sB_{\cV}$. 
	Therefore, for any $h\in   B_\cV (\BX)$, $\Lambda_t(\cdot,u;h)\in B_{\cV}(\BX)$ for each $u\in\BU$.  
\end{lemma}
\begin{proof} 
	It is easy to see that  $\Lambda^\e_t(\cdot,u;h^\e)$ is uniformly bounded below.	Note that  by \eqref{entropy},
	$$\ba{ll}\ns\Lambda_t^\e(x,u;h^\e)\ad=\sup_{\nu\in P(\BX)}\(\int_{\BX}h^\e d\nu-\e \cR(\nu\Vert q_t^\e(\cdot;x,u))\)\\
	\ns\ad\leq\sup_{\nu\in P(\BX)}\(\int_{\BX}\lambda\cV d\nu-\e \cR(\nu\Vert q_t^\e(\cdot;x,u))\)+\sup_{\nu\in P(\BX)}\int_{\BX}(h^\e-\lambda\cV) d\nu\\
	\ns\ad\leq \Lambda_t^\e(x,u;\lambda\cV)+\sup_{\BX}[h^\e-\lambda\cV]\ea$$
	By (A3),
	\bel{heun}\ba{ll}\ds\limsup_{|x|\rightarrow\infty}\frac{\sup_\e\Lambda_t^\e(x,u;h^\e)}{\cV(x)}\leq  \limsup_{|x|\rightarrow\infty}\frac1{\cV(x)}\(\sup_\e\Lambda_t^\e(x,u;\lambda\cV)+\sup_{\BX}[\sup_\e h^\e-\lambda\cV]\)\leq \lambda K_u\ea\eel
	By the arbitrariness of $\lambda>0$, it follows that $\{\Lambda_t^\e(\cdot,u;h^\e)\}\in   \sB_\cV $ for each $u\in \BU$.

\end{proof}

\ms

 Now we are ready to present the Hamiltonians used in our paper. Define $\cA^\e_t[\cdot],~\cA_t[\cdot]:  B_\cV (\BX)\mapsto M(\BX)$ by
$$\cA_{t}^\e[h](x):=\inf_{u\in\BU}\[f_{t,t}(x,u)+\Lambda^\e_t(x,u;h)\]\q
\text{and}\q\cA_{t}[h](x):=\inf_{u\in\BU}\[f_{t,t}(x,u)+\Lambda_t(x,u;h)\].$$

The following lemma will guarantee that $\cA_t^\e$ and $\cA_t$ map $  B_\cV (\BX)$ into $  B_\cV (\BX)$ under the following assumption.
\ss

  {\bf Assumption (B)}: {\it 
	For each fixed $u\in\BU$, $f_{\t,t}(\cdot,u),~  g_\t(\cdot)\in   B_\cV (\BX)$. 
	For each fixed $i\in\BX$, $f_{t,t}(i,\cdot)$ is continuous and inf-compact.

}
\begin{lemma} Under Assumptions {\rm (A)} and {\rm (B)}, for any $h\in   B_\cV (\BX)$, $\cA_{t}^\e[h],~\cA_{t}[h]\in   B_\cV (\BX)$.

\end{lemma}
\begin{proof} Since $f_{t,t}$ and $h$ are bounded below, so are $\cA_{t}[h]$ and $\cA_{t}^\e[h]$ by their definitions. Since $$\cA_{t}[h](x)\leq f_{t,t}(x,u_0)+\Lambda_t(x,u_0;h), \text{ for some } u_0\in\BU,$$ by Lemma \ref{corlambdah} and Assumption (B), $\cA_{t}[h](x)\in   B_\cV (\BX).$ Similarly we have $\cA^\e_{t}[h](x)\in   B_\cV (\BX).$ Moreover the infimums can be attained by  Assumptions (A) and (B). 
	
\end{proof}

Given any $h\in   B_\cV (\BX)$,  define 
$$\Box\eta_t^\e(\cdot;h):x\mapsto \argmin_{u\in\BU}[f_{t,t}(x,u)+\Lambda^\e_t(x,u;h)]\subset\BU.$$
If $\Box\eta_t^\e(x;h)\neq\emptyset$ for any $x\in\BX$,  we say $\eta^\e_t(\cdot;h)$ is a choice of $\Box\eta_t^\e(\cdot;h)$  if 
$$\eta^\e_t(x;h)\in\Box\eta_t^\e(x;h)\q\text{ for any }x\in \BX.$$
We write it as $\eta^\e_t(\cdot;h)\in \Box\eta^\e_t(\cdot;h)$. Since $\BX$ is a countable-stated space,  $\eta^\e_t(\cdot;h)$ is naturally measurable.  Similarly we can define $\Box\eta_t$ and its one choice $\eta_t$.
\ms

Define $\cH^\e_{\t,t}[\cdot],~\cH_{\t,t}[\cdot]:  B_\cV (\BX)\mapsto M(\BX\times\BU)$ by
$$\cH_{\t,t}^\e[h](x,u):=f_{\t,t}(x,u)+\Lambda^\e_t(x,u;h)\q
\text{and}\q\cH_{\t,t}[h](x,u):=f_{\t,t}(x,u)+\Lambda_t(x,u;h).$$
It is easy to see that 
$$\cA_{t}^\e[h](x)=\inf_{u\in\BU}\cH^\e_{t,t}[h](x,u)\q\text{and}\q\cA_{t}[h](x)=\inf_{u\in\BU}\cH_{t,t}[h](x,u).$$

\ms

\ms

From their  definitions, we know that $\cH_{\t,t}^\e,~\cH_{\t,t}$ will map $  B_\cV (\BX)$ into $M(\BX)$. We raise   the following assumption  to guarantee $\cH_{\t,t}^\e[h],~\cH_{\t,t}[h]\in   B_\cV (\BX)$ for any  $h\in   B_\cV (\BX)$ and fixed $u\in\BU$.

{\bf Assumption (C)} {\it
	Let $$B_{t,\lambda}(x):=\{u\in\BU: f_{t,t}(x,u)\leq \lambda \cV(x)\}.$$  There exists  constants  $\lambda_0>0$ and $K_1$   such that for any $\lambda\in(0,\l_0)$ and $h\in   B_\cV (\BX)$,
	$$\limsup_{|x|\rightarrow\infty}\frac{\sup_{u\in B_{t,\lambda}(x)}f_{\t,t}(x,u)}{ \lambda \cV(x)}\leq  K_1$$
	and 
	$$\limsup_{|x|\rightarrow\infty}\frac{\sup_{0<\e<\e_0}\sup_{u\in B_{t,\lambda}(x)}\Lambda^\e_t(x,u;\lambda\cV)}{\lambda \cV(x)}\leq K_1.$$

}
\begin{remark}
	{\rm We can see that 
		$$\limsup_{|x|\rightarrow\infty}\frac{\sup_{u\in B_{t,\lambda}(x)}\Lambda_t(x,u;\lambda\cV)}{\lambda \cV(x)}\leq K_1$$
		Moreover, since $f_{t,t}\in   B_\cV (\BX)$, 
		any $u_0\in\BX$ belongs to $B_{t,\lambda}(x)$ if $|x|$ is large. Thus 
		(A3) is a consequence of Assumptions (B) and (C).	
	}
\end{remark}

\begin{lemma} \label{BBcV}Under   Assumptions {\rm (A)}, {\rm (B)} and {\rm (C)}, the followings are true.\ss
	
	{\rm (1)} for any $h\in   B_\cV (\BX)$, $\eta_t(\cdot,h)\in\Box\eta_t(\cdot,h)$ and  $\eta^\e_t(\cdot,h)\in\Box\eta^\e_t(\cdot,h)$,
	$$f_{\t,t}(\cdot,\eta_t(\cdot;h)),f_{\t,t}(\cdot,\eta^\e_t(\cdot;h))\in   B_\cV (\BX). $$
	
	{\rm (2)} for any $h_1,h_2\in   B_\cV (\BX)$,  $\eta_t(\cdot,h_2)\in\Box\eta_t(\cdot,h_2)$ and  $\eta^\e_t(\cdot,h_2)\in\Box\eta^\e_t(\cdot,h_2)$, $$\Lambda^\e_t(\cdot,\eta^\e_t(\cdot;h_2);h_1),~\Lambda_t(\cdot,\eta_t(\cdot;h_2);h_1)\in   B_\cV (\BX).$$

\end{lemma}

\begin{proof} 
	
	(1)	Recall the definitions  $$B_{t,\lambda}(x):=\{u\in\BU:f_{t,t}(x,u)\leq \lambda \cV(x)\}$$
	and 
	$$B_t(x):=\{u\in\BU: f_{t,t}(x,u)\leq  f_{t,t}(x,u_0)+\Lambda_t(x,u_0;h)-\inf_\BX h+\d\}.$$
	Since for each $u\in \BU$, $f_{t,t}(\cdot,u)\in   B_\cV (\BX)$, $B_t(x)\subset B_{t,\lambda}(x)$ for large $|x|$.
	By the definition of $\eta_t$ and $\eta_t^\e$, it follows that $$\eta_t^\e(x;h),\eta_t(x;h)\in B_{t,\lambda}(x) \text{ when $x$ is large }.$$
	As a consequence, for each $\lambda\in(0,\lambda_0)$,
	$$\limsup_{|x|\rightarrow\infty}\frac{f_{\t,t}(x,\eta_t(x;h))}{\cV(x)}\leq \limsup_{|x|\rightarrow\infty}\frac{\sup_{u\in B_{t,\lambda}(x)}f_{\t,t}(x,u)}{ \cV(x)}\leq \lambda K_1.$$
	By the arbitrariness of $\lambda\in(0,\lambda_0)$, it follows that $f_{\t,t}(x,\eta_t(x;h))\in   B_\cV (\BX)$.
	Similarly, we can prove that $f_{\t,t}(x,\eta^\e_t(x;h))\in   B_\cV (\BX)$.\ms
	
	(2)	Let $h_1,h_2\in   B_\cV (\BX)$. Then
	$$\eta_t(x;h_2),\eta^\e_t(x;h_2)\in B_{t,\lambda}(x)\text{ when $|x|$ is large. }$$
	
		By the definition of $ \Lambda_t(\cdot,u;\lambda\cV)$, we have 
		$-I_t(z;x,u)\leq \Lambda_t(x,u;\lambda\cV)-\lambda\cV(z).$
		Note that 
		\bel{useI}\Lambda_t(x,u;h)=\sup_{z\in\BX}[h(z)-I(z;x,u)]\leq \sup_{z\in\BX}[h(z)-\lambda\cV(z)]+\Lambda_t(x,u;\lambda\cV).\eel
	Therefore, $$\ba{ll}\Lambda_t(x,\eta_t(x;h_2);h_1)\ad\leq \sup_{z\in\BX}[h_1(z)-\lambda\cV(z)]+\Lambda_t(x,\eta_t(x;h_2);\lambda\cV) \\
	\ns\ad\leq \sup_{z\in\BX}[h_1(z)-\lambda\cV(z)]+\sup_{u\in B_{t,\lambda}(x)}\Lambda_t(x,u;\lambda\cV)\ea$$
	and for any $\lambda\in(0,\lambda_0)$, by  Assumption (C),
	$$\limsup_{|x|\rightarrow\infty}\frac{\Lambda_t(x,\eta_t(x;h_2);h_1)}{\cV(x)}\leq \lambda K_1.$$
	By the arbitrariness of $\lambda\in(0,\lambda_0)$, it follows that $\Lambda_t(x,\eta_t(x;h_2);h_1)\in \sB_\cV $ for any $h_1,h_2\in \sB_\cV $. Similarly, the result holds for $\Lambda^\e_t(x,\eta^\e_t(x;h_2);h_1)\in   B_\cV (\BX)$.

\end{proof}

\section{Time-inconsistent Equilibrium}\label{sec:TE}
In this section, we will derive the time-inconsistent equilibrium strategy step by step. The section will be divided into several subsections.
\subsection{Optimal Control for 1-step Transition}
In this subsection, we will review the 1-step optimal control problem with risk-sensitive cost. Consider $\{X^\e_1, X^\e_2\}$ with controlled transition probability 
$$\dbP(X^\e_2=j|X^\e_1=i;u)=q_t^\e(j;i,u)$$
Let 
$$\Lambda^\e (x,u;h)=\e\log\dbE\(\exp[\e^{-1}h(X_2)]\big|X_1=x;u\).$$

Given some function $\hat f :\BX\times \BU\mapsto\BR$ and $\hat g:\BX\mapsto\BR$, define the cost function
$$\hat V(x)=\hat J(x;\tilde u):=\e\log\dbE\(\exp\[\e^{-1}\big(\hat f(X_1,u(X_1))+\hat g(X_2)\big)\]\big|X_1=x\).$$

{\bf Problem-(CON)}: to find a $\tilde u^*\in\cU$ such that 
$$\hat J(x;\tilde u^*)=\inf_{ \tilde u\in\cU}\hat J(x;\tilde u).$$

By the definition of $\Lambda^\e$, we have
$$\hat J(x;\tilde u)=\hat f(x,u(x))+\Lambda^\e(x,u(x);\hat g).$$
As a result,
$$\hat V(x)=\inf_{u\in\BU}[\hat f(x,u)+\Lambda^\e(x,u;\hat g)]\text{ and } u^*(x)\in\argmin_{u\in\BU}[\hat f(x,u)+\Lambda^\e(x,u;\hat g)]$$
Note that the optimal strategy $u^*(\cdot)$ might not be unique. The existence of $\tilde u^*=u^*(\cdot)$ will be guaranteed  by the assumptions in the proof.

\ms
\subsection{Time-inconsistent Strategy}\label{eequilirium}

Now we  are ready to introduce the recursion process of finding the time-inconsistent equilibria. We start with the last step first and move backward to the first step. 
\ms 

{\it $T$-th step strategy.} In the last step, the control is determined by  solving a classical optimal control problem with discounting factor being $\t=T$.

{\bf Problem-$T$}: to find $\tilde u^{\e,*}_T\in\cU$ such that
$$J^\e_{T,T}(x;\tilde u_{T}^{\e,*})=\inf_{\tilde u\in\cU} J^\e_{T,T}(x;\tilde u).$$
By the definition of $\Lambda^\e_T$, one can see 
$$J^\e_{T,T}(x;\tilde u)=f_{T,T}(x,u)+\Lambda_T^\e(x,u;g_T).$$
Thus the optimal control in this step is in the following feedback form
 \bel{nequi}\ba{ll} u^{\e,*}_T(x)\ad\in\argmin_{u\in \BU}\left\{f_{T,T}(x,u)+\Lambda_T^\e(x,u;g_T)\right\}=\Box\eta_T^\e(x;g_T).\ea\eel
 By Assumption (B), the optimal feedback control must exist. The value function is 
$$V^\e_{T}(x)=J^\e_{T,T}(x;\tilde u_{T}^{\e,*})=\inf_{u\in \BU}\left\{f_{T,T}(x,u)+\Lambda_T^\e(x,u;g_T)\right\}=\cA_T^\e[g_T](x).$$
While the minimum point is not unique, let $\eta^\e_t(\cdot;g_T)$ be a choice of $\Box\eta_t^\e(\cdot;g_T)$.  We choose
\bel{nequiunique}u_T^{\e,*}(x)=\eta^\e_T(x;g_T).\eel

Given the optimal control we find this step, now for any $\t\in\dbT$, let
\bel{ThetaN}\Theta^\e_{\t,T}(x):=f_{\t,T}(x,\eta^\e_T(x;g_T))+\Lambda_T^\e(x,\eta^\e_T(x;g_T);g_\t)=\cH^\e_{\t,T}[g_\t](x,\eta^\e_T(x;g_T)).\eel
It is easy to see that
\bel{ThetaJ}\Theta^\e_{\t,T}(x)=J^\e_{\t,T}(x;\eta^\e_T(x;g_T)),\eel
i.e. $\Theta^\e_{\t,T}$ is the value of the cost function at time $T$ if we use the discounting factor $\t$ and the  feed-back control $\eta^\e_T(x;g_T)$. Note that $$\Theta^\e_{T,T}(x)=\inf_{\tilde u\in\cU} J^\e_{T,T}(x;\tilde u).$$

\ms

{\it $(T-1)$-th step strategy.}  In the $(T-1)$th step,  we know  $T$-step strategy is $\tilde u_T^*=\eta^\e_T(\cdot;g_T)$  defined by \eqref{nequiunique} under discounting factor $\t=T$. While in this step, the strategy is based on  the new discounting factor $\t=T-1$.  Thus we are solving the following optimal control problem.\ss

{\bf Problem-$(T-1)$}: to find $\tilde u^{\e,*}_{T-1}\in\cU$ such that
$$J^\e_{T-1,T-1}(x;\tilde u_{T-1}^{\e,*}\oplus \tilde u_{T}^{\e,*})=\inf_{\tilde u\in\cU} J^\e_{T-1,T-1}(x;\tilde u\oplus \tilde u^{\e,*}_T).$$

Note that 
$$J^\e_{T-1,T-1}(x;\tilde u\oplus \tilde u^{\e,*}_T)=f_{T-1,T-1}(x,u(x))+\Lambda^\e_{T-1}(x,u(x);J^\e_{T-1,T}(x;\tilde u_T^{\e,*}))$$
and by \eqref{ThetaJ}, $$\Theta^\e_{T-1,T}(x)=J^\e_{T-1,T}(x,\tilde u_T^{\e,*}).$$

Similarly we can take $\eta^\e_{T-1}(\cdot;\Theta^\e_{T-1,T})$, a possible  choice of $\Box \eta^\e_{T-1}(\cdot;\Theta^\e_{T-1,T})$ and let
\bel{n-1equi}\ba{ll} u^{\e,*}_{T-1}(x)\ad=\eta^\e_{T-1}(x;\Theta^\e_{T-1,T}).\ea\eel
 The value function 
$$V^\e_{T-1}(x)=J^\e_{T-1,T-1}(x;\tilde u_{T-1}^{\e,*}\oplus \tilde u_{T}^{\e,*})=\inf_{u\in \BU}\left\{f_{T-1,T-1}(x,u)+\Lambda_{T-1}^\e(x,u;\Theta^\e_{T-1,T})\right\}=\cA^\e_{T-1}[\Theta^\e_{T-1,T}](x).$$
Here  $\Lambda_{T-1}^\e(x,u;\Theta^\e_{T-1,T})$ is well-defined since $\Theta_{T-1,T}^\e\in   B_\cV (\BX)$ by Lemma \ref{BBcV}.

Now for any $\t\in\dbT$, let
\bel{ThetaN-1}\ba{ll}\Theta^\e_{\t,T-1}(x)\ad:=f_{\t,T-1}(x,\eta^\e_{T-1}(x;\Theta^\e_{T-1,T}))+\Lambda_{T-1}^\e(x,\eta^\e_{T-1}(x;\Theta^\e_{T-1,T});\Theta^\e_{T-1,T})\\
\ns\ad=\cH^\e_{\t,T-1}[\Theta^\e_{\t,T}](x,\eta^\e_{T-1}(x;\Theta^\e_{T-1;T})).\ea\eel
It is easy to see that
$$\Theta^\e_{\t,T-1}(x)=J^\e_{\t,T-1}(x;\tilde u_{T-1}^{\e,*}\oplus\tilde u_{T}^{\e,*}).$$
\ms

{\it $t$-th step strategy.}  Before $t$th step, it has been   already identified that $\tilde u_{t+1,\dbT}^*=\eta^\e_{t+1}(\cdot;\Theta^\e_{t+1,t+2})\oplus\cdots\oplus\eta^\e_{T}(\cdot;g_T)$. In this step, we are using the new discounting factor $\t=t$. Thus we are solving the following optimal control problem.

{\bf Problem-$t$}: to find $\tilde u^{\e,*}_t\in\cU$ such that
$$J^\e_{t,t}(x;\tilde u_{t}^*\oplus \tilde u_{t+1,\dbT}^{\e,*})=\inf_{\tilde u\in\cU} J^\e_{t,t}(x;\tilde u\oplus \tilde u_{t+1,\dbT}^{\e,*})$$

Similarly  we can take one choice among the possible multiple choices that
\bel{kequi}\ba{ll} u^{\e,*}_{t}(x)\ad=\eta^\e_{t}(x;\Theta^\e_{t;t+1})\ea\eel
and the value function 
$$V^\e_{t}(x)=J^\e_{t,t}(x;\tilde u_{t}^{\e,*}\oplus \tilde u_{t+1,\dbT}^{\e,*})=\inf_{u\in \BU}\left\{f_{t,t}(x,u)+\Lambda_{t}^\e(x,u;\Theta^\e_{t,t+1})\right\}=\cA_t^\e[\Theta^\e_{t,t+1}].$$

Now for any $\t\in\dbT$, let
\bel{Thetak}\Theta^\e_{\t,t}(x):=f_{\t,t}(x,\eta^\e_{t}(x;\Theta^\e_{t,t+1}))+\Lambda_{T-1}^\e(x,\eta^\e_{t}(x;\Theta^\e_{t,t+1});\Theta^\e_{\t,t+1})=\cH^\e_{\t,t}[\Theta_{\t,k+1}](x,\eta^\e_t(x;\Theta^\e_{t,t+1})).\eel
It is easy to see that
$$\Theta^\e_{\t,t}(x)=J^\e_{\t,t}(x;\tilde u_t^{\e,*}\oplus\cdots\oplus\tilde u_T^{\e,*}).$$
\ms

By recursively repeating such process until the first step, we get a $T$-step strategy $\eta_\dbT^\e=\eta^\e_1\oplus\cdots\oplus\eta^\e_T$  and a sequence of functions $\{\Theta^\e_{\t,t}:(\t,t)\in\dbT\times\dbT\}$ by the following recursions,
 
\bel{HJBmain}\left\{\ba{ll}\ns\ad\Theta^\e_{\t,t}(x)=\cH_{\t,t}^\e[\Theta^\e_{\t,t+1}](x,\eta^\e_t(x;\Theta_{t,t+1}^\e)),\q \t,t\in\dbT\\ 
\ns\ad \eta^\e_t(\cdot;\Theta_{t,t+1}^\e)\in\Box\eta^\e_t(\cdot;\Theta_{t,t+1}^\e)\\ 
\ns\ad \Theta^\e_{\t,T+1}(x)=g_\t(x).\ea\right.\eel

Similarly, we can construct 
 $T$-step strategy $\eta_\dbT=\eta_1\oplus\cdots\oplus\eta_T$ and a sequence of functions $\{\Theta_{\t,t}:(\t,t)\in\dbT\times\dbT\}$ by the following recursions,

\bel{HJBmain-0}\left\{\ba{ll}\ns\ad\Theta_{\t;t}(x)=\cH_{\t,t}[\Theta_{\t,t+1}](x,\eta_t(x;\Theta_{t,t+1})),\q \t,t\in\dbT\\ 
\ns\ad \eta_t(\cdot;\Theta_{t,t+1})\in\Box\eta_t(\cdot;\Theta_{t,t+1})\\ 
\ns\ad \Theta_{\t;T+1}(x)=g_\t(x).\ea\right.\eel

\begin{remark}
	{\rm (1) One can see that the construction of $\eta^\e_{\dbT}$ ($\eta^\e_{\dbT}$)  is in a reverse order. Moreover, if the choices $\eta_t^\e$ ($\eta_t$)  changes, 	$\eta_s^\e$ ($\eta_s$) for $s<t$ have to change correspondingly. \ms
		
		(2)	If $f_{\t,t}$ and $g_\t$ is independent of $\t$, i.e. the  time-consistent case, then $\cH^\e_{\t,t}[h](x,u)=\cA^\e_{t}[h](x)$ for any $u\in\Box\eta^\e(x;\Theta^\e_{t,t+1})$. Thus $\Theta^\e_{\t,t}=\Theta^\e_{t,t}$ for any $\t\in\dbT$
		and the recursion for the value function is   $V^\e_t(x)=\Theta^\e_{t,t}(x)$
		$$V^\e_t=\cA^\e_t[V^\e_{t+1}], \q\text{with}\q V^\e_{T+1}=g.$$
	One can see that the Hamiltonian recursion is independent of the choice of the optimal control	in each step now.

	}
\end{remark}

\bigskip

Now we are ready to introduce our first main theorem.

\begin{theorem} Under Assumptions {\rm (A),(B)} and {\rm (C)}, the followings hold.
	
	{\rm (1)} For any choice of $\eta^\e_\dbT:=\eta^\e_1\oplus\cdots\oplus\eta^\e_T$ constructed in Section \ref{eequilirium} ,
	the recursive sequence	$\{\Theta^\e_{\t,t}(\cdot):(\t,t)\in\dbT\times\dbT\}$  from \eqref{HJBmain} is well-defined in $  B_\cV (\BX)$. Moreover $\eta^\e_\dbT$ is a time-inconsistent $\e$-risk-sensitive  equilibrium.

	{\rm (2)} For any choice of $\eta_\dbT:=\eta_1\oplus\cdots\oplus\eta_T$ constructed in Section \ref{eequilirium} ,
	the recursive sequence	$\{\Theta_{\t,t}(\cdot):(\t,t)\in\dbT\times\dbT\}$  from \eqref{HJBmain-0} is well-defined in $  B_\cV (\BX)$. Moreover $\eta_\dbT$ is a time-inconsistent risk-sensitive equilibrium.
	
	{\rm (3)} Any time-inconsistent $\e$-risk-sensitive equilibrium   $\eta_\dbT$, coupled with $\Theta^\e_{\t,t}(x)=J^\e_{\t,t}(x;\eta^\e_{t,\dbT})$, solves \eqref{HJBmain}. 
	
   {\rm (4)}	Any time-inconsistent risk-sensitive equilibrium  $\eta_\dbT$, coupled with $\Theta_{\t,t}(x)=J_{\t,t}(x;\eta_{t,\dbT})$, solves \eqref{HJBmain-0}.

\end{theorem}

\begin{proof}
	
	(1) and (2).  By  Lemma  \ref{BBcV}, $\{\Theta^\e_{\t,t}(\cdot):(\t,t)\in\dbT\times\dbT\}$ and $\{\Theta_{\t,t}(\cdot):(\t,t)\in\dbT\times\dbT\}$ are well-defined in $  B_\cV (\BX)$.
	
	 By the construction process of $\eta^\e_\dbT$, one can see that 
	$$\ba{ll}\Theta^\e_{t,t}(x)\ad=\cH^\e_{t,t}[\Theta_{t,t+1}](x,\eta^\e_t(x;\Theta^\e_{t,t+1}))\\
	\ns\ad=\cA_t[\Theta_{t,t+1}](x)\\
	\ns\ad=\inf_{u\in\BU}[f_{t,t}(x,u)+\Lambda_t^\e(x,u;\Theta_{t,t+1}^\e)].\ea$$
	and $\Theta^\e_{t,t}(x)=J^\e_{t,t}(x;\tilde u^{\e,*}_{t,\dbT})$.
	The optimality \eqref{opt1}  holds directly. Thus $\eta^\e_\dbT$ is a time-inconsistent risk-sensitive $\e$-equilibrium. Similar argument can be applied to $\eta_\dbT$ as well.

	(3) and (4). If $\eta_\dbT$  is a time-inconsistent risk-sensitive equilibrium strategy, by the optimality \eqref{opt2}, $$u_t^*(\cdot)=\eta_t(\cdot;\Theta_{t,t+1})\in\Box\eta_t(\cdot;\Theta_{t,t+1}).$$
	By $\Theta_{\t,t}(x)=J_{\t,t}(x;\eta_{t,\dbT})$, it is easy to see that 
	$$\Theta_{\t;t}(x)=\cH_{\t,t}[\Theta_{\t,t+1}](x,\eta_t(x;\Theta_{t,t+1})),\q\t,t\in\dbT.$$
	Thus $\{\eta_t,\Theta_{\t,t}\}$ solves \eqref{HJBmain-0}. The similar results holds for $\eta_\dbT^\e$.
\end{proof}

\section{The Convergence of $\e$-equilibria}\label{sec:con}

In this section, we focus on the convergence of $\e$-equilibria as $\e\rightarrow 0^+$, i.e. whether the solutions of \eqref{HJBmain} converges to some solution of \eqref{HJBmain-0} as $\e\rightarrow 0^+$. We need the following two lemmas.

\begin{lemma}\label{uniformconvergencelemma}Under Assumptions {\rm (A)},     if $\{h^\e\}\in \sB_\cV$ and $h^\e\rightarrow h$ point-wisely,
	then $\Lambda_t^\e(x,u;h^\e)$	converges to $\Lambda_t(x,u;h)$ uniformly on any compact compact set of $\BU$. 
	\end{lemma}

\begin{proof} 
	Let $u^\e\rightarrow u$. Without loss of generality, we assume that  that the uniform lower bound of $h^\e$ is 1 and write $\hat h=\sup_\e h^\e$ and $C_m=\sup_{x>m}(\hat h-\lambda_0\cV)$.  It is easy to see that
	$\lim_{m\rightarrow\infty}C_m=-\infty.$   Let $$D_m:=\limsup_{\e\rightarrow0^+}\e\log\(\sum_{z>m}\exp\left\{\e^{-1} \hat h(z)\right\}q^\e_{t}(z;x,u^\e) \).$$
	
	Note that 
	\bel{uniforminter-0}\ba{ll}\ds \e\log\(\sum_{z\geq m }\exp\left\{\e^{-1} \hat h(z)\right\}q^\e_{t}(z;x,u^\e) \)\ad\leq \e\log\(\sum_{z\geq m }\exp\left\{\e^{-1}  \lambda_0\cV(z)\right\}q^\e_{t}(z;x,u^\e) \)+C_m\\
	\ns\ad
	\leq \Lambda^\e(x,u^\e;\lambda_0\cV_0)+C_m, \ea\eel
Thus we have
		\bel{uniforminter}\lim_{m\rightarrow\infty}D_m=-\infty, \eel
Note that for any $m>0$, we have 
	\bel{uniformboundh}\limsup_{\e\rightarrow0+}\sup_{|x|\leq m}|h^\e-h|=0.\eel
	By \eqref{uniformboundh}, 
	we have  
	$$\ba{ll}\ns\ad\limsup_{\e\rightarrow0+}\Lambda_t^\e(x,u^\e;h^\e)\\
	\ns\ad \q=\limsup_{\e\rightarrow0+}\e\log\(\sum_{z\in\BX}\exp\{\e^{-1} h^\e(z)q^\e_{t}(z;x,u^\e)\}\)\\
	\ns\ad \q\leq \limsup_{\e\rightarrow0+}\e\log\(\sum_{z\leq M}\exp\left\{\e^{-1} (h^\e(x)-h(x))\right\}\exp\{h(z)\}q^\e_{t}(z;x,u^\e)+ \sum_{z>M}\exp\{\e^{-1} \hat h(z)\}q^\e_{t}(z;x,u^\e)\)\\
	\ns\ad\leq \limsup_{\e\rightarrow0+}\(\sup_{|x|\leq M}|h^\e-h|+\Lambda_t(x,u;h)\)\bigvee \limsup_{\e\rightarrow0+}\e\log\(\sum_{z>M}\exp\{\e^{-1} \hat h(z)\}q^\e_{t}(z;x,u^\e)\)\\
	\ns\ad\leq \Lambda_t(x,u;h)\vee D_m\ea$$
	On the other hand,
	$$\ba{ll}\ns\ad\liminf_{\e\rightarrow0+}\Lambda_t^\e(x,u^\e;h^\e)\\
	\ns\ad \q\geq\liminf_{\e\rightarrow0+}\e\log\(\sum_{z\leq M}\exp\{\e^{-1} h^\e(z)\}q^\e_{t}(z;x,u^\e)\}\)\\
	\ns\ad \q\geq \liminf_{\e\rightarrow0+}\e\log\(\sum_{z\leq M}\exp\left\{-\e^{-1} (h^\e(z)-h(z))\right\}\exp\{h(z)\}q^\e_{t}(z;x,u^\e)\)\\
	\ns\ad \q\geq -\limsup_{\e\rightarrow0+}\sup_{|x|\leq M}|h^\e-h|+\liminf_{\e\rightarrow0+}\e\log\(\sum_{z\in \BX}\exp\{h(z)\}q^\e_{t}(z;x,u^\e)-\sum_{z> M}\exp\{\hat h(z)\}q^\e_{t}(z;x,u^\e)\)\\
	\ns\ad\q\geq \Lambda_t(x,u;h)\vee D_m\ea$$
By the arbitrariness of $m$ and \eqref{uniforminter}, it follows that $$\lim_{\e\rightarrow0+}\Lambda_t^\e(x,u^\e;h^\e)=\Lambda_t(x,u;h).$$ 		
Since $u^\e\rightarrow u$ is arbitrary,  it is equivalent to  
	\bel{uniformconvergen h}\lim_{\e\downarrow0}\Lambda_t^\e(x,u;h^\e)=\Lambda_t(x,u;h),\q\text{ uniformly on any compact set of }\BU.\eel	
\end{proof}

The following lemma concerns with a stability result of the Hamiltonians.
\begin{lemma}\label{strategiesconvergence}	
	Under Assumptions {\rm (A), (B)} and {\rm (C)}, the followings hold.\ss
	
	{\rm (1)} Suppose  $\{h^\e\}\in   \sB_\cV $. Then  $\{\eta^\e_t(\cdot;h^\e)\}_{0<\e<\e_0}$ is compact in point-wise convergence sense and the limit of any convergent subsequence (as $\e\rightarrow0$) belongs to  $\Box\eta_t(\cdot;h)$. \ss
	
	{\rm (2)}  Let  $\{h_1^\e\},\{h_2^\e\}_\e\in   \sB_\cV $  and $h_1^\e\rightarrow h_1$ and  $h_2^\e\rightarrow h_2$ point-wisely.   For any convergent subsequence $\{\eta^{\e_n}_t(\cdot;h_2^{\e_n})\}$ ($\e_n\rightarrow0^+$) with limit $\eta_t^0(\cdot;h_2)$
	\bel{Hconvergence}\lim_{n\rightarrow\infty}\cH^{\e_n}_{\t,t}[h_1^{\e_n}](x;\eta^{\e_n}_t(x;h_2^{\e_n}))=\cH_{\t,t}[h_1](x;\eta_t(x;h_2)) \text{ for any }x\in\BX.\eel
	Moreover 
	$\{\cH^{\e}_{\t,t}[h_1^{\e}](\cdot;\eta^{\e}_t(\cdot;h_2^{\e}))\}\in   \sB_\cV .$
\end{lemma}

\begin{proof}(1) 
	Recall $$\Box\eta_t^\e(x;h^\e)=\argmin_{u\in\BU}[f_{t,t}(x,u)+\Lambda_t^\e(x,u;h^\e)].$$
Let $$B_t(x):=\{u\in\BU:f_{t,t}(x,u)\leq f_{t,t}(x,u_0)+\sup_\e\Lambda_t^\e(x,u_0;\lambda\cV)+\sup_{\BX}[\sup_\e h^\e-\lambda\cV]-\inf_{\e} \inf_\BX h^\e\}$$
By \eqref{uniformconvergen h} and \eqref{heun}, for any fixed $x$, we can see that $\Box\eta_t^\e(x;h^\e)\subset B_t(x)$ and $B_t(x)$ is compact. Thus for any sequence of choices $\{\eta_t^{\e_n}(x;h^{\e_n})\}$ ($\e_n\rightarrow0$), there exists a convergent subsequence with limit $\eta_t^0(x;h)$.		 
Note that
 $$|\cA^\e_{t}[h^\e](x)-\cA_{t}[h](x)|\leq \sup_{u\in B_t(x)}|\Lambda_t^\e(x,u;h^\e)-\Lambda(x,u;h)|$$
Therefore by \eqref{uniformconvergen h},
$$\lim_{n\rightarrow\infty}\cA^{\e_n}_{t}[h^{\e_n}](x)=\cA_{t}[h](x).$$
Moreover,
$$\ba{ll}\ds\cA_{t}[h](x,u)\ad=\lim_{n\rightarrow\infty}\cA^{\e_n}_{t}[h^{\e_n}](x,u)|\\
\ns\ad =\lim_{n\rightarrow\infty}[f_{t,t}(x,\eta_t^{\e_n}(x;h^{\e_n}))-f_{t,t}(x,\eta_t(x;h))+\Lambda^\e_{t}(x,\eta_t^{\e_n}(x;h^{\e_n});h^\e)-\Lambda_{t}(x,\eta_t(x;h);h)]\\
\ns\ad\q+f_{t,t}(x,\eta_t(x;h))+\Lambda_{t}(x,\eta_t(x;h);h)\\
\ns\ad=f_{t,t}(x,\eta_t(x;h))+\Lambda_{t}(x,\eta_t(x;h);h).\ea$$
The last step holds since \eqref{uniformconvergen h} and $f_{t,t}$ is continuous. Thus $\eta_t(x;h)$ the minimum point of $\cA_{t}[h](x,\cdot)$ for fixed $x\in\BX$.  

Since $\BX$ has only countable many states, by the classical diagonalization method, one can extract a convergent subsequence $\eta_t^{\e_n}$ such that the convergence is true for any $x\in\BX$, i.e.
$$\lim_{\e_n\rightarrow 0}\eta_t^{\e_n}(x;h)=\eta_t(x;h),\q\text{for any }x\in\BX.$$

(2)
Note that $$\ba{ll}\cH^{\e_n}_{\t,t}[h_1^{\e_n}](x;\eta^{\e_n}_t(x;h_2^{\e_n}))\ad=f_{\t,t}(x,\eta^{\e_n}_t(x;h_2^{\e_n}))+\Lambda_t^{\e_n}(x,\eta^{\e_n}_t(x;h_2^{\e_n});h_1^{\e_n})\\
\ns\ad =f_{\t,t}(x,\eta^{\e_n}_t(x;h_2^{\e_n}))+\Lambda_t(x,\eta^{\e_n}_t(x;h_2^{\e_n});h_1)\\
\ns\ad\q+\Lambda_t^{\e_n}(x,\eta^{\e_n}_t(x;h_2^{\e_n});h_1^{\e_n})-\Lambda_t(x,\eta^{\e_n}_t(x;h_2^{\e_n});h_1)\ea$$
Since $f_{\t,t}$ is continuous,
$$\lim_{n\rightarrow 0}f_{\t,t}(x,\eta^{\e_n}_t(x;h_2^{\e_n}))=f_{\t,t}(x,\eta_t(x;h_2)).$$
By (A2), 
$$\lim_{n\rightarrow\infty}\Lambda_t(x,\eta^{\e_n}_t(x;h_2^{\e_n});h_1)=\Lambda_t(x,\eta_t(x;h_2);h_1).$$
By \eqref{uniformconvergen h},
$$\ba{ll}\ad\lim_{n\rightarrow 0}|\Lambda_t^{\e_n}(x,\eta^{\e_n}_t(x;h_2^{\e_n});h_1^{\e_n})-\Lambda_t(x,\eta^{\e_n}_t(x;h_2^{\e_n});h_1)|\\
\ns\ad\q\leq \lim_{n\rightarrow 0}\sup_{u\in B_t(x)}|\Lambda_t^{\e_n}(x,u;h_1^{\e_n})-\Lambda_t(x,u;h_1)|=0.\ea$$
Therefore, thus \eqref{Hconvergence} holds.

It is easy to see that $\{\sup_\e\cH^{\e}_{\t,t}[h_1^{\e}](x;\eta^{\e}_t(x;h_2^{\e}))\}_\e$ is uniformly bounded below. Now we will prove that $\sup_\e\cH^{\e}_{\t,t}[h_1^{\e}](x;\eta^{\e}_t(x;h_2^{\e}))\in   B_\cV (\BX)$.\ss

By (A3), $\eta_t^\e(x;h^\e_2)\in B'_t(x)$ 
where  $$B'_t(x):=\{u\in\BU:f_{t,t}(x,u)\leq f_{t,t}(x,u_0)+\sup_\e\Lambda_t^\e(x,u_0;\lambda\cV)+\sup_{\BX}[\sup_\e h_2^\e-\lambda\cV]-\inf_{\e} \inf_\BX h_2^\e\}$$ and  $B'_t(x)\subset B_{t,\lambda}(x)$ when $|x|$ is large,  where $ B_{t,\lambda}(x)$ is defined from Assumption (C) $$B_{t,\l}(x)=\{u\in\BU:f_{t,t}(x,u)\leq \lambda \cV(x) \}.$$
Simple calculation yields
$$\ba{ll}\cH^{\e}_{\t,t}[h_1^{\e}](x;\eta^{\e}_t(x;h_2^{\e}))\ad\leq \sup_{u\in B_{t,\l}(x)}[f_{\t,t}(x,u)+\Lambda_t^\e(x,u;h_1^\e)]\\
\ns\ad \leq  \sup_{u\in B_{t,\l}(x)}\(f_{\t,t}(x,u)+\Lambda_t^\e(x,u;\lambda\cV)+\sup_{z}[h_1^\e(z)-\lambda\cV(z)]\)\ea$$
By Assumption (B),
$$\limsup_{|x|\rightarrow\infty }\frac1{\cV(x)}\sup_\e\cH^{\e}_{\t,t}[h_1^{\e}](x;\eta^{\e}_t(x;h_2^{\e}))\leq K_1\lambda.$$
By the arbitrariness of $\lambda$, we have $\{\cH^{\e}_{\t,t}[h_1^{\e}](x;\eta^{\e}_t(x;h_2^{\e}))\}\in   \sB_\cV $.

\end{proof}

Now we are ready to establish the convergence of time-inconsistent  $\e$-equilibria  to time-inconsistent   equilibrium as $\e\rightarrow0^+$.

\begin{theorem}Under Assumptions {\rm (A),~(B)} and {\rm (C)}, 
	 as $\e\rightarrow 0^+$,  the sequence of time-inconsistent  $\e$-equilibria    $\{\eta_\dbT^\e\}$ is compact (in pointwise convergence sense) and the limit $\eta_\dbT$ of any convergent subsequence $\{\eta^{\e_n}_\dbT\}$ is  a time-inconsistent  equilibrium strategy. At the same time, $\{\Theta^{\e_n}_{\t,t}\}$ defined in \eqref{opt1} using $\eta^{\e_n}_\dbT$ converges to $\{\Theta_{\t,t}\}$ defined in \eqref{opt2} using $\eta_\dbT$ in $(  B_\cV (\BX),w)$.
\end{theorem}

\begin{proof}
	At $N$th step, we take a subsequence 
	$\{\eta_N^{\e_n}(\cdot;g_N)\}$ with limit $\eta_N$. Note that 
	$$\Theta^{\e_n}_{\t,N}(x)=\cH_{\t,t}^{\e_n}[g_\t](x;\eta_N^{\e_n}(x;g_N)).$$
	By Lemma \ref{strategiesconvergence}, we know that 
	$\Theta_{\t,N}^{\e_n}(x)$ is uniformly bounded below and $\sup_{\e_n}\Theta^{\e_n}_{\t,N}\in B_{\cV}(\BX)$ with limit 	$\Theta_{\t,N}^{0}$ in point-wise sense. By Lemma \ref{lemmapoint}, $\Theta_{\t,N}^{\e_n}$ converges to $\Theta_{\t,N}^{0}$ in  $(  B_\cV (\BX),w)$.
	
	At $(N-1)$th step, we take a subsequence of $\{\eta_{N-1}^{\e_n}(\cdot;\Theta_{N-1,N})\}$ (still written as the same sequence) with limit $\eta_{N-1}$.
	Note that 
	$$\Theta_{\t,N-1}^{\e_n}(x)=\cH_{\t,t}^{\e_n}[\Theta^{\e_n}_{\t,N}](x;\eta_{N-1}^{\e_n}(x;\Theta^{\e_n}_{N-1,N})).$$
	By Lemma \ref{strategiesconvergence},  
	 $\{\Theta_{\t,N-1}^{\e_n}\in B_{\cV}(\BX)\}\in  \sB_\cV $  and it converges to 	$\Theta_{\t,N-1}(x)$ point-wisely as $n\rightarrow\infty$. Thus $\{\Theta_{\t,N-1}^{\e_n}$ converges to $\Theta_{\t,N-1}$ in $(  B_\cV (\BX),w)$.	We repeat such process until the first step.  Then the proof is complete.
\end{proof}

The following corollary is obvious.

\begin{corollary}
	{\rm (1)} Under Assumptions {\rm (A),(B)} and {\rm (C)}, if the solution $\Theta$ of \eqref{opt2} is unique, then $\Theta_{\t,t}^\e$ converges to $\Theta_{\t,t}$ in $(  B_\cV (\BX),w)$, i.e.
	$$\lim_{\e\rightarrow0^+}w(\Theta_{\t,t}^\e,\Theta_{\t,t})=0.$$
	
	{\rm (2)} Under Assumptions {\rm (A)} and {\rm (B)}, if the cost functional is independent of the a discounting factor $\t$, i.e. time-consistent case, the solution $\Theta$ ($\Theta^\e$ resp.) is independent of the choices $\eta$ ($\eta^\e$ resp.) as well. As a result the solution $\Theta$  of \eqref{opt2} is unique and   $V_t^\e=\Theta^\e_{t,t}$ converges to $V_t=\Theta_{t,t}$ in $(  B_\cV (\BX),w)$ for each fixed $t$.
\end{corollary}

This similar results for stochastic diffusion in $\dbR^d$ is called  vanishing viscosity procedure in \cite{Flem2006}. While in our case, we are dealing with a discrete-time, countable-stated MDP. Compared to that of \cite{Flem2006}, our result has it own interesting feature because our problem  is time-inconsistent.

\section{ Illustrative Examples}\label{sec:exp}

In this section, we will present two illustrative examples.  First, we present an example in which the assumptions are possible to be verified.
\begin{example}\label{example1}	{\rm

	 Consider  a sequence of random variables defined by $$X^\e_{t+1}=X^\e_t+u+\xi^\e_t.$$
		where the control $u$ is taken in $\BU=\{-1,1\}$ and  the distribution function of $\xi_t^\e$ is
		$$\dbP(\xi^\e_t=x)=\left\{\ba{ll}\kappa\exp\{-\e^{-1}|x|^2\},\qq \ad\text{if } x\neq 0\\ [2mm]
		\ns\ds 1-\kappa\sum_{z\neq 0}\exp\{-\e^{-1}|z|^2\},\qq \ad\text{if } x=0.\ea \right.$$
		for some small $\kappa>0$.		
Simple calculation yields that
$$I(z;x,u)=(z-x-u)^2\text{ and
}\Lambda(x,u;h)=\sup_{z\in\BX}[h(z)-(z-x-u)^2)].$$	 
 Let $\cV(x)=|x|^2$. Take $\e_0$ small, (A1) holds. Since $\BU$ is compact, (A2) holds. Note that 
 $$\ba{ll}\Lambda^\e(x,u;\l\cV)\ad=\e\log\left\{\kappa\sum_{z\neq 0}\exp\{\e^{-1}(\lambda|x+z|^2-|z|^2)\}+\exp\{\e^{-1}\lambda|x|^2\}(1-\kappa\sum_{z\neq 0}\exp\{-\e^{-1}|z|^2\})\right\}\\
 \ns\ad\leq  \max\(\sup_{z}[\lambda |x+z|^2-|z|^2],\lambda|x|^2\)\leq \lambda K_1 \cV(x)\ea$$
 Therefore (A3) holds.

Let $f_{\t,t}(\cdot,u),g_\t(\cdot)\in   B_\cV (\BX)$. Since $\BU$ is compact, Assumption (B)  and (C) are trivial because of (A3).
Because the infimum or supremum can be attained, simple calculation shows that the Hamiltonians are
$$\left\{\ba{ll}\ns\ad \cH_{\t,t}[h](x,u)=f_{\t,t}(x,u)+\max_{z\in\BX}[h(z)-(z-x-u)^2)]\\
\ns\ad \cA_t[h](x)=\min_{u\in\BU}\(f_{t,t}(x,u)+\max_{z\in\BX}[h(z)-(z-x-u)^2]\).\ea\right.$$
We can easily get the recursion sequence defined in \eqref{HJBmain-0}.

If $f$ and $g$ are independent of the discounting factor $\t$. Then the value function $V_t$ satisfies
$$\left\{\ba{ll}\ad V_t(x)=\min_{u\in\BU}\[f_{t}(x,u)+\max_{z\in\BX}[V_{t+1}(z)-(z-x-u)^2]\],\\
\ns\ad V_T(x)=g(x).\ea\right.$$
This is the  time-inconsistent case which is equivalent to discrete  min-max control problem.\ss

 Now we assume the problem is $c_{\t,t}$ is exponential discounting, i.e. $c_{\t,t}(x,u)=\lambda^{t-\t}c(x,u)$ for some $\l\in(0,1)$. Suppose the problem was time-consistent with a (global) optimal strategy $\pi$. Due to the non-linear structure in the cost functional, in general 
 $J_{t-1,t}(x;\pi)\neq \l J_{t,t}(x;\pi)$. For example, one can see that 
 $J_{\t,T}(x)=e^{T-\t}g(x)$ and 
 $$\ba{ll}\l J_{T-1,T-1}(x;\pi)\ad=\l\(c(x,u)+\inf_{z}\{\l g(z)-(z-x-u)^2\}\) \\
 	\ns\ad\neq \l c(x,u)+\inf_{z}\{\l^2 g(z)-(z-x-u)^2\}=J_{T-2,T-1}(x;\pi).\ea $$
This is to say
$$\ba{ll}V_{t-1}(x)\ad=\inf_{u}\(c(x,u)+\inf_z\big[  J_{t-1,t}(x;\pi)-(z-x-u)^2\big] \)\\
\ns\ad\neq\inf_{u}\(c(x,u)+\inf_z\big[ \l J_{t,t}(x,\pi)-(z-x-u)^2\big] \)\\
\ns\ad =\inf_{u}\(c(x,u)+\inf_z\big[ \l V_t(x,\pi)-(z-x-u)^2\big] \). \ea$$
This contradicts to the global optimality of $\pi$ we supposed. Thus it is impossible for us to find an optimal strategy even if $\t$ is in an exponential form. Such result matches what we claimed previously in introduction and  motivates us to investigate time-inconsistent problems.
 
}

\end{example}

\begin{example}
	{\rm In a regime-switching  financial model,  the stock market may switch between two states (i.e. bull and bear) under some probability law. We assume that the investors' actions can effect the transition of stock market between different states, while might bring some serious consequence with rare probability. For example, due to the actions $u$ taken by investors, there appears  a third    state  (i.e. crisis) with a rare  occurrence  rate  which is proportional to the parameter $\e$. When $\e$ is small,  the rare occurrence may lead  neglectable effect to general investors, but a strong effect to risk-sensitive ones.
	
Let $X^u_t$, the state of stock market, be  a controlled Markov chain  with state space $\{1,2,3\}$.  The transition probability follows that 
$$q^\e(1;1,u)=1-p_1(u)-e^{-\frac {\lambda(1,u)}{\e}},q^\e(2;1,u)=p_1(u)-e^{-\frac {\lambda(1,u)}{\e}}, q^\e(3;1,u)=2e^{-\frac {\lambda(1,u)}{\e}}$$
$$q^\e(1;2,u)=p_2(u)-e^{-\frac {\lambda(2,u)}{\e}},q^\e(2;2,u)=1-p_2(u)-e^{-\frac {\lambda(2,u)}{\e}}, q^\e(3;2,u)=2e^{-\frac {\lambda(2,u)}{\e}}$$ 
$$q^\e(1;3,u)=1-p_3(u)-e^{-\frac {\lambda(3,u)}{\e}},q^\e(2;3,u)=p_3(u)-e^{-\frac {\lambda(3,u)}{\e}}, q^\e(3;3,u)=2e^{-\frac {\lambda(3,u)}{\e}}.$$
where  $\lambda(i,u)\geq 0$ and $u\in\{0,1\}$  represents whether the the investor takes action to the system.  Observed from the transition law, the first two states are general  and the third one is rarely existed. If $\lambda(i,0)=0$ and $\lambda (1,1),\lambda(2,1)>0$, the rare occurrence of  state  3 is   because of the investor's action. Now let suppose that  risk-sensitive decision-maker makes their decisions with a cost functional similar to \eqref{cost-0}.

When $\e\rightarrow 0^+$, simple calculation implies that 
$$\left\{\ba{ll}\ad  I(3;x,u)=\lambda (x,u),\qq I(1;3,u)=I(2;3,u)=0;\\[2mm]
\ad I(z;x,u)=0, \qq\text{if } x,z\in\{1,2\}.\ea\right.$$
Then one can see that
$$\ba{ll}\Lambda(x,u;h)=\max[h(1),h(2),h(3)-\lambda(x,u)]
\ea$$

If the third state was not existed, i.e. $\lambda(x,u)=0$ for any $x,u$, one can conclude that
$\Lambda(x,u;h)=\max[h(1),h(2),h(3)]$ is independent of $u$. This essentially says  risk-sensitive investor at time $t$ takes actions $u$ only  to minimize the cost
the cost $f_{t,t}(x,u)$  at the step only. Due to the possibility of rare state, risk-sensitive investors  have to change their strategies accordingly.

}
\end{example}

\section{Concluding Remarks}\label{sec:conrem}

We have explored the time-inconsistent risk-sensitive MDPs with countable-stated state space. Due to the time-inconsistency of the risk-sensitive cost function, the theory on the time-inconsistent equilibria and the convergence of value function  as $\e\rightarrow0^+$  have some unique interesting features, e.g. the convergence of $\e$-equilibria are required for the convergence of value functions.  Therefore, our results enrich the general theory of risk-sensitive MDPs and the time-inconsistent control problems. For our time-inconsistent risk-sensitive MDPs, a Hamiltonian recursion for each $\e>0$ has been derived and the convergence for the solution sequences as $\e\rightarrow0^+$ has been proved. An example is presented to show our assumptions are  general.

We still can see that the theory is  in its infancy and it is possible to be improved in several aspects. For example, can we conclude the similar results for general state space like $\BX=\BR^d$? The main difficulty  lies in the first-order regularity of the viscosity solutions  of non-linear PDEs. We hope to report it in the other paper.\ms

\paragraph{Acknowledgements}
The author is gratitude for the two anonymous referees for their helpful suggestions which have improved the manuscript a lot. The author would also like to thank Professor Fran\c{c}ois Dufour for his valuable comments on the early version of the manuscript.

\bibliographystyle{amsplain}

	\end{document}